\documentclass[11pt]{amsart}
\usepackage{amscd,amsmath,amsthm,amssymb}
\usepackage{pstcol,pst-plot,pst-3d}
\usepackage{lmodern,pst-node}
\usepackage{multicol}
\usepackage{epic,eepic}
\usepackage{lineno}
\usepackage{amsfonts,amssymb,amscd,amsmath,enumerate,verbatim}
\psset{unit=0.7cm,linewidth=0.8pt,arrowsize=2.5pt 4}

\newpsstyle{fatline}{linewidth=1.5pt}
\newpsstyle{fyp}{fillstyle=solid,fillcolor=verylight}
\definecolor{verylight}{gray}{0.97}
\definecolor{light}{gray}{0.9}
\definecolor{medium}{gray}{0.85}
\definecolor{dark}{gray}{0.6}
 \def\NZQ{\Bbb}               
 \def\NN{{\NZQ N}}

 \def\PP{{\NZQ P}}

 \def\MI{{\mathcal I}}

\def\MP{{\mathcal P}}

 \def\opn#1#2{\def#1{\operatorname{#2}}} 
 \opn\chara{char} \opn\length{\ell} \opn\pd{pd} \opn\rk{rk}
 \opn\projdim{proj\,dim} \opn\injdim{inj\,dim} \opn\rank{rank}
 \opn\depth{depth} \opn\grade{grade} \opn\height{height}
 \opn\embdim{emb\,dim} \opn\codim{codim}
 
 \opn\Tr{Tr} \opn\bigrank{big\,rank}
 \opn\superheight{superheight}\opn\lcm{lcm}
 \opn\trdeg{tr\,deg}
 \opn\reg{reg} \opn\lreg{lreg} \opn\ini{in} \opn\lpd{lpd}
 \opn\size{size} \opn\sdepth{sdepth}
 \opn\link{link}\opn\fdepth{fdepth}\opn\lex{lex}
 \opn\div{div} \opn\Div{Div} \opn\cl{cl} \opn\Cl{Cl}
 \opn\Spec{Spec} \opn\Supp{Supp} \opn\supp{supp} \opn\Sing{Sing}
 \opn\Ass{Ass} \opn\Min{Min}\opn\Mon{Mon}
 \opn\Ann{Ann} \opn\Rad{Rad} \opn\Soc{Soc}
 \opn\Im{Im} \opn\Ker{Ker} \opn\Coker{Coker} \opn\Am{Am}
 \opn\Hom{Hom} \opn\Tor{Tor} \opn\Ext{Ext} \opn\End{End}
 \opn\Aut{Aut} \opn\id{id}
 
 \opn\nat{nat}
 \opn\pff{pf}
 \opn\Pf{Pf} \opn\GL{GL} \opn\SL{SL} \opn\mod{mod} \opn\ord{ord}
 \opn\Gin{Gin} \opn\Hilb{Hilb}\opn\sort{sort}\opn\Syz{Syz}
 \opn\aff{aff} \opn
 \con{conv} \opn\relint{relint} \opn\st{st}
 \opn\lk{lk} \opn\cn{cn} \opn\core{core} \opn\vol{vol}  \opn\inp{inp} \opn\nilpot{nilpot}
 \opn\link{link} \opn\star{star}\opn\lex{lex}\opn\set{set} \opn\rev{rev} \opn\ini{in}
 \opn\gr{gr}

 \def\pot#1#2{#1[\kern-0.28ex[#2]\kern-0.28ex]}

 \opn\dirlim{\underrightarrow{\lim}}
 \opn\inivlim{\underleftarrow{\lim}}
 \def\Implies{\ifmmode\Longrightarrow \else
         \unskip${}\Longrightarrow{}$\ignorespaces\fi}
 \def\implies{\ifmmode\Rightarrow \else
         \unskip${}\Rightarrow{}$\ignorespaces\fi}
 \def\iff{\ifmmode\Longleftrightarrow \else
         \unskip${}\Longleftrightarrow{}$\ignorespaces\fi}

 \let\:=\colon
 \newtheorem{Theorem}{Theorem}[section]
 
 \newtheorem{Corollary}[Theorem]{Corollary}
 
 \newtheorem{Remark}[Theorem]{Remark}
 
 \newtheorem{Example}[Theorem]{Example}
 
 \newtheorem{Definition}[Theorem]{Definition}

 \let\epsilon\varepsilon
 \let\kappa=\varkappa
 \def\qed{\ifhmode\textqed\fi
       \ifmmode\ifinner\quad\qedsymbol\else\dispqed\fi\fi}
 \def\textqed{\unskip\nobreak\penalty50
        \hskip2em\hbox{}\nobreak\hfil\qedsymbol
        \parfillskip=0pt \finalhyphendemerits=0}
 \def\dispqed{\rlap{\qquad\qedsymbol}}
 
 \opn\dis{dis}
 \def\pnt{{\raise0.5mm\hbox{\large\bf.}}}
 
 \opn\Lex{Lex}

\begin{document}

 \title {Rank bounded Hibi subrings for planar distributive lattices}

 \author {Rida Irfan, Nadia Shoukat}

\address{Nadia Shoukat, Abdus Salam School of Mathematical Sciences, GC University,
Lahore. 68-B, New Muslim Town, Lahore 54600, Pakistan} \email{nadiashoukat85@yahoo.com}
\address{ Rida Irfan, Department of Mathematics, COMSATS University Islamabad, Sahiwal Campus, Pakistan} \email{ridairfan@cuisahiwal.edu.pk}

\thanks{The first author was supported by the Higher Education Commission of Pakistan and the Abdus Salam School of Mathematical Sciences, Lahore, Pakistan. The authors are deeply grateful to Prof. Viviana Ene for her supervision to accomplish this work.}

\begin{abstract}
Let $L$ be a distributive lattice and $R[L]$ the associated Hibi ring. We show that if $L$ is  planar, then any bounded Hibi subring of $R[L]$ has a quadratic  Gr\"obner basis.   We characterize all planar distributive lattices $L$ for which any proper rank bounded Hibi subring of $R[L]$ has a linear resolution. Moreover, if $R[L]$ is linearly related for a lattice $L$, we find all the rank bounded Hibi subrings of $R[L]$ which are linearly related too.
\end{abstract}

\subjclass[2010]{13D02, 13C05, 05E40, 13P10}

\keywords{Rank bounded Hibi Subrings, Linear resolution, Linear syzygies}

\maketitle

\section{Introduction}

Hibi rings and their defining ideals are attached in a natural way to finite distributive lattices. They were introduced by Hibi in
\cite{Hibi}.

Let $L$ be a finite distributive lattice and let $P$ be the subposet of $L$ which consists of the join-irreducible elements of $L.$ Then,
by a famous theorem of Birkhoff \cite{B}, it follows that $L$ is isomorphic to the lattice  $\MI(P)$ of the poset ideals of
$P.$ We recall that an element $a\in L$ is called join-irreducible if $a$ is not the minimal element of $L$ and if $a=b\vee c$ where  $b,c
\in L$, then $a=b$ or $a=c,$ in other words, $a$ does not admit a proper decomposition as a join of elements of $L.$ Let us also recall
the definition of the poset ideal. A subset $\alpha\subseteq P$ is called a poset ideal if it has the following property: for any $a\in \alpha $
and $ b\in P, $ if $b\leq a,$ then $b\in \alpha.$ For a comprehensive study of finite lattices we refer to \cite{B,Sta1}.

Let us assume that $P$ consists of $n$ elements, say $P=\{p_1,\ldots,p_n\}$. Let $A=K[t,x_1,\ldots,x_n]$ be the polynomial ring in $n+1$
variables over a field $K.$   The Hibi ring of $L=\MI(P)$ is the $K$--subalgebra $R[L]$ of $A$ generated over $K$ by the monomials
$u_{\alpha}=t\prod_{p_i\in \alpha}x_i,$ with $\alpha\in L.$ In \cite{Hibi}, Hibi showed that $R[L]$ is an algebra with straightening laws
on $L$ over $K$ (ASL, in brief). For an extensive survey on ASL, we refer the reader to \cite{Eis}. Let $B=K[\{x_\alpha: \alpha \in L\}]$
be the polynomial ring in the indeterminates indexed by the elements of $L.$ Then, the defining ideal $I_L\subset B$ of the toric ring $R[L]$ is generated
by the straightening relations, namely
$$I_L=(x_\alpha x_\beta-x_{\alpha\cap \beta}x_{\alpha\cup \beta}: \alpha,\beta \in L, \alpha,\beta \text{ incomparable}).$$
$I_L$ is called the Hibi ideal  or the join-meet ideal of $L.$

In the same paper \cite{Hibi}, Hibi showed that $R[L]$ is a Cohen-Macaulay normal domain of $\dim R[L]=|P|+1.$ In the last decades, many authors have investigated various properties and invariants of Hibi rings; see, for example, \cite{AHH,EHH,EHHSara,EHSara,ERQ,HHNach,Q2}. Generalizations of Hibi rings are studied in the papers \cite{Hetal, EHM}.

Much less is known about the so-called rank bounded Hibi subrings introduced in \cite{AHH}.

Let $L$ be a finite distributive lattice and let $p,q$ be integers such that $0\leq p<q\leq \rank L.$ The rank bounded Hibi subring
$R[p,q;L]\subseteq R[L]$ is the $K$--subalgebra of $R[L]$ generated over $K$ by all the monomials $u_{\alpha}=t\prod_{p_i\in \alpha}x_i,$ with $\alpha \in L$ and $p\leq \rank \alpha \leq q.$ In \cite{AHH}, it was shown that if a Hibi ring $R[L]$ possesses a quadratic Gr\"obner basis with respect to a rank lexicographic order, then all the rank bounded Hibi subrings have the same property.

In this paper, we study rank bounded Hibi subrings for planar distributive lattices.
 The paper is organized as follows. In Section~\ref{one}, we show that any rank bounded Hibi subring of $R[L]$, where $L$ is a planar
distributive lattice, has a quadratic Gr\"obner basis. In order to prove this theorem, we interpret $R[p,q;L]$ as an edge ring of a suitable
bipartite graph which has only cycles of length $4.$ As a consequence, we derive that $R[p,q;L]$ is a Cohen-Macaulay normal domain.

In Section~\ref{two}, we study several homological properties of rank bounded Hibi subrings. In Theorem~\ref{linear}, we characterize the planar distributive lattices $L$ with the property that every proper rank bounded Hibi subring of $R[L]$ has a linear resolution. In particular, we see that if $R[L]$ has a linear resolution, then every rank bounded Hibi subring of $R[L]$ has a linear resolution. As it follows from Example~\ref{notlinearrel}, if $R[L]$ is linearly related then it does not necessarily follow that any rank bounded Hibi subring of $R[L]$ inherits the same property. However, given a lattice $L$ such that $R[L]$ is linearly related, we may find all the rank bounded Hibi subrings of $R[L]$ which are linearly related too; see Theorem~\ref{linrelth}.

\section{The Gr\"obner basis}
\label{one}

Let $\NN^2$ be the infinite distributive lattice with the partial order defined as follows: $(i,j)\leq (k,\ell)$ if $i\leq k$ and $j
\leq\ell.$ A planar distributive lattice $L$ is a finite sublattice of $\NN^2$ which contains $(0,0)$ and has the following property: if $(i,j)
, (k,\ell)\in L$ and $(i,j)<(k,\ell)$, then there exists a chain of elements in $L:$ $$(i,j)=(i_0,j_0)<(i_1,j_1)<\cdots < (i_t,j_t)=(k,
\ell)$$
such that $i_{s+1}+j_{s+1}=i_s+j_s+1$ for all $0\leq s\leq t-1.$
If $a,b \in \NN^2$ with $a \leq b$,
then the set $[a,b]= \{ c \in \NN^2|\; a \leq c \leq b\}$ is an interval of $\NN^2$. The interval $C=[a,b]$ with $b=a+(1,1)$ is called a  cell  of $\NN^2$. Any planar distributive lattice may be viewed as a convex polyomino, as it was observed  in \cite{EHH}. For more information about polyominoes and their ideals we refer to \cite{Q,EHH}.

In what follows, we consider only simple planar distributive lattices, that is, lattices $L$ with the property that, for any $0<\ell<\rank L,$ there exist at least two elements of $L$ of rank $\ell.$

\begin{Definition}
Let $L$ be a planar distributive lattice and let $p,q$ be integers such that $0\leq p<q\leq \rank L.$
 The $K$--subalgebra $R[p,q;L]$ of
$R[L]$ generated by all the monomials $u_\alpha$ with $p\leq \rank \alpha\leq q$ is called a rank bounded Hibi subring of $R[L].$
\end{Definition}

If $p=0,$ we call $R[p,q;L] $ a rank upper-bounded Hibi subring. Similarly, if $q=\rank L,$ we call $R[p,q;L] $ a rank lower-bounded Hibi subring.

In this section, we show that any rank bounded Hibi subring has a quadratic Gr\"obner basis.

Let $L$ be a planar distributive lattice. The elements of $L$ are lattice points $(i,j)$ in the plane with $0\leq i\leq m$ and $0\leq j\leq n$ for some
positive integers $m,n$ with $m+n=\rank L.$ Then, the Hibi ring $R[L]$ may be viewed as the edge ring of a bipartite graph $G_L$ which admits the vertex bipartition $\{s_1,\ldots,s_m\}\cup \{t_1,\ldots,t_n\}$ and whose edge set is $E(G_L)=\{\{s_i,t_j\}: (i,j)\in L\}.$ Note that the generator $s_it_j$ of the edge ring corresponds to an element in $L$ whose rank is $i+j.$ Let
$0\leq p<q\leq \rank L$ and $R[p,q;L]$ a rank bounded Hibi subring of $R[L].$ Then $R[p,q;L]$ coincides with the subring of the edge ring
$K[\{s_it_j: 0\leq i\leq m, 0\leq j\leq n\}]$ which is generated by all the monomials $s_it_j$ with $p\leq i+j\leq q.$

\begin{Example} {\em
In Figure~\ref{forproof} we have a lattice of rank $5+4=9$ whose elements are the lattice points contained in the polygon bounded by the fat polygonal line. The fat points in the figure correspond to the generators of the  subring of $R[L]$ bounded by $p=3$ and $q=7.$
The bounded subring of $R[L]$ has $14$ generators as an algebra over $K$.
}
\end{Example}

\begin{figure}[]
\begin{center}
\psset{unit=0.6cm}
\begin{pspicture}(0,0)(5,4)

\pspolygon(0,0)(5,0)(5,4)(0,4)
\psline(0,1)(5,1)
\psline(0,2)(5,2)
\psline(0,3)(5,3)
\psline(1,0)(1,4)
\psline(2,0)(2,4)
\psline(3,0)(3,4)
\psline(4,0)(4,4)

\rput(3,0){$\bullet$}
\rput(2,1){$\bullet$}
\rput(1,2){$\bullet$}
\rput(3,1){$\bullet$}
\rput(1,3){$\bullet$}
\rput(3,2){$\bullet$}
\rput(2,3){$\bullet$}
\rput(4,2){$\bullet$}
\rput(2,1){$\bullet$}
\rput(2,2){$\bullet$}
\rput(3,3){$\bullet$}
\rput(3,3){$\bullet$}
\rput(2,4){$\bullet$}
\rput(5,2){$\bullet$}
\rput(4,3){$\bullet$}
\rput(3,4){$\bullet$}
\psline[linestyle=dashed](1,2)(3,0)
\psline[linestyle=dashed](3,4)(5,2)
\pspolygon[linewidth=1.8pt](0,0)(3,0)(3,2)(5,2)(5,4)(2,4)(2,3)(1,3)(1,2)(0,2)
\rput(-0.6,-0.1){$(0,0)$}
\rput(-0.6,4.1){$(0,4)$}
\rput(5.6,-0.1){$(5,0)$}
\rput(5.6,4.1){$(5,4)$}
\end{pspicture}
\end{center}
\caption{Representation of a rank bounded Hibi subring}
\label{forproof}
\end{figure}
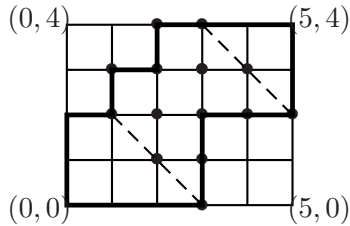

\begin{Theorem}\label{GB} Let $L$ be a planar distributive lattice and $p,q$ integers with $0\leq p<q\leq \rank L.$
The defining ideal of the rank bounded Hibi subring $R[p,q;L]$ has a quadratic Gr\"obner basis.
\end{Theorem}

\begin{proof}
As we have already observed,  $R[p,q;L]$ may be identified with the subring $K[\{s_it_j: (i,j)\in L,  p\leq i+j\leq q\}]$ of the edge ring
$R[L]=K[\{s_it_j: (i,j)\in L\}].$ In \cite{OH}, it was shown that an edge ring of a bipartite graph $G$ has a
quadratic Gr\"obner basis (with respect to a suitable monomial order) if and only if every cycle of $G$ of length $>4$ has a chord.
We follow the ideas of the proof of \cite[Theorem 2.1]{EHH}. Let $G$ be the bipartite graph with edges $\{s_i,t_j\}$, where $p\leq i+j\leq q.$ An even cycle of length $2r$ in $G$ is a sequence of edges of $G$ which correspond to a sequence of lattice points in the plane,  say $a_1,\ldots,a_{2r},$ with $a_{2k-1}=(i_k,j_k)$ and
$a_{2k}=(i_{k+1},j_k)$ for $k=1,\ldots,r$, where $i_{r+1}=i_1,$ $i_k\neq i_{\ell}$ and $j_k\neq j_{\ell}$ for $k\neq \ell, k,\ell\leq r.$
By \cite[Lemma 2.2]{EHH}, it follows that there exist integers $c,d$ with $1\leq c,d\leq r,$ $d\neq c,c+1$ such that either
$i_c<i_d<i_{c+1}$ or $i_{c+1}<i_d<i_c.$ Let us choose $i_c<i_d<i_{c+1}$. The other case may be discussed similarly. Since
$a_{2c-1}=(i_c,j_c)$ and $a_{2c}=(i_{c+1},j_c)$ correspond to edges of $G,$ we have: $p\leq i_{c}+j_c< i_{c+1}+j_c\leq q.$ As $i_c<i_d<i_{c+1}$,
it follows that $p\leq i_d+j_c\leq q$ which implies that $(i_d,j_c)$ corresponds to a chord in our cycle of $G.$
\end{proof}

\section{Properties of rank bounded Hibi subrings}
\label{two}

Theorem~\ref{GB} and its proof have important consequences. In fact, in view of the theorem proved in \cite{OH}, the proof of Theorem~\ref{GB} shows that the binomials of the quadratic Gr\"obner basis of the defining ideal of $R[p,q;L]$ are differences of squarefree monomials of degree $2.$ This immediately implies the following.

\begin{Corollary}
Let $L$ be a planar distributive lattice and $p,q$ integers with $0\leq p<q\leq \rank L.$ Then $R[p,q;L]$ is a normal Cohen-Macaulay domain.
\end{Corollary}

\begin{proof}
By a theorem of Sturmfels \cite{St95}, since the defining ideal of $R[p,q;L]$ has a squarefree initial ideal, it follows that
$R[p,q;L]$ is a normal domain. The Cohen-Macaulay property follows from a classical  theorem of Hochster \cite{Ho72}.
\end{proof}

\begin{Remark}{\em
In \cite[Section 2]{AHH} it was shown that if $L$ is a chain ladder, then $R[L]$ and any bounded subring of $R[L]$ have a lexicographic
quadratic Gr\"obner basis. Our Theorem~\ref{GB} does not impose any additional  condition on the planar distributive lattice $L$ to derive that any rank bounded subring of $R[L]$ has a quadratic Gr\"obner basis.
}
\end{Remark}

\begin{Remark}{\em
Let $L$ be a planar distributive lattice, $p,q$ integers with $0\leq p<q\leq \rank L,$ and $T=K[\{y_{ij}: (i,j)\in L, p\leq i+j\leq q\}]$ the polynomial ring over $K.$ By the proof of Theorem~\ref{GB}, the defining ideal of $R[p,q;L]$ is the binomial ideal of $T$ generated
by the quadratic binomials $y_{ij}y_{k\ell}-y_{i\ell}y_{kj}$, where $(i,j),(k,\ell)\in L$, $p\leq i+j, k+\ell\leq q$ and $p\leq i+\ell, j+k\leq q.$

On the other hand, let us observe that one may consider the collection $\MP$ of all the cells $[a,a+(1,1)]\subset L$ with $a=(i,j)\in L, a+(1,1)\in L,$ and such that $p\leq i+j<i+j+2\leq q.$ This is obviously a convex polyomino. Indeed $\MP$ is row convex and column convex. We give only the argument for row convexity since column convexity works similarly. Let $[a,a+(1,1)], [b,b+(1,1)]\in \MP,$ where $a=(i,j), b=(k,j), i<k.$
Then, if $[c,c+(1,1)]\in \NN^2$ is a cell with $c=(\ell,j)$ where $i\leq \ell\leq k, $ then we have
\[
p\leq i+j\leq \ell+j<\ell+j+2\leq k+j+2\leq q.
\]
This shows that $[c,c+(1,1)]\in \MP.$

Let $T^\prime\subset T$ be the polynomial ring in the variables $y_{ij}$, where $(i,j)$ is a vertex of $\MP.$ Then, according to the proof of \cite[Theorem 2.1]{EHH}, the polyomino ideal $I_{\MP}$ is generated by the quadratic binomials $y_{ij}y_{k\ell}-y_{i\ell}y_{kj}$, where  $(i,j),(k,\ell)\in L,$  $p\leq i+j, k+\ell\leq q$ and $p\leq i+\ell, j+k\leq q.$ Thus, the defining ideal of $R[p,q;L]$ is nothing else but $I_{\MP}T.$

This simple observation will be very useful in our further study.
}
\end{Remark}

Theorem~\ref{GB} has another obvious consequence.

\begin{Corollary}
Let $L$ be a planar distributive lattice and $p,q$ integers with $0\leq p<q\leq \rank L.$ Then $R[p,q;L]$ is Koszul.
\end{Corollary}

\begin{proof}
It is a classical result that a standard graded $K$--algebra is Koszul if its defining ideal has a quadratic Gr\"obner basis. For a proof we refer to \cite[Theorem 6.7]{EH}.
\end{proof}

\begin{Remark}{\em
Let $\MP$ be the collection of all the cells $[a,a+(1,1)]\subset L$ with $a=(i,j)\in L, a+(1,1)\in L,$ and such that $p\leq i+j<i+j+2\leq q.$ Then
$$\dim R[p,q;L]=(\text{number of vertices} ~(i,j) ~\text{in}~ L ~\text{such that} ~p\leq i+j\leq q)-\text{number of cells in}~ \MP.$$

This fact is the direct consequence of \cite[Corollary 2.3]{Q}.
}
\end{Remark}

In what follows, we are interested in relating some   homological properties of rank bounded Hibi subrings to the corresponding properties of the Hibi ring.

Let $L$ be a planar distributive lattice and let $I_L\subset K[\{y_{ij}: (i,j)\in L\}]$ be the defining ideal of $R[L].$ For any
rank bounded Hibi subring $R[p,q;L]$ of $R[L]$, we denote by $I^{p,q}_L$ the defining ideal of $R[p,q;L]$ which is contained in the polynomial ring $K[\{y_{ij}:(i,j)\in L, p\leq i+j\leq q\}]$.

\begin{Theorem}\label{linear}
Let $L$ be a planar distributive lattice. Then the defining ideal of any proper rank bounded Hibi subring of $R[L]$ has a linear resolution if and only if one of the following conditions holds:
\begin{itemize}
	\item [(i)] $I_L$ has a linear resolution, that is, $L=\MI(P)$ where $P$ is the direct sum of a chain and an isolated
element.
\item [(ii)] $L=\MI(P)$ where $P$ is the poset $p_1<p_2, q_1<q_2, q_1<p_2$.
\item [(iii)] $L=\MI(P)$ where $P$ is the poset $p_1<p_2, q_1<q_2, p_1<q_2$.
\end{itemize}
\end{Theorem}

\begin{proof}
Let $L$ be a planar distributive lattice. Suppose that the defining ideal of any proper rank bounded Hibi subring of $R[L]$ has a linear resolution. In Remark $3.3$, it is shown that the defining ideal of any rank bounded Hibi subring is nothing else but the ideal $I_{\MP}$, where $\MP$ is a convex polyomino. We employ here \cite[Theorem 4.1]{EHH} which states that, for a convex polyomino $\MP$, $I_{\MP}$ has a linear resolution if and only if $\MP$ consists of either a row of cells or a column of cells, that is, $\MP$ is of one of the forms displayed in Figure~\ref{Polylinres}:

\begin{figure}[hbt]
\begin{center}
\psset{unit=0.4cm}
\begin{pspicture}(0,-1)(9,3)

\pspolygon(0,0)(2,0)(2,1)(0,1)
\pspolygon(3,0)(4,0)(4,1)(3,1)
\psline(1,0)(1,1)
\psline[linestyle=dashed](2,0)(3,0)
\psline[linestyle=dashed](2,1)(3,1)
\pspolygon(8,-1)(9,-1)(9,0)(8,0)
\pspolygon(8,1)(9,1)(9,3)(8,3)
\psline(8,2)(9,2)
\psline[linestyle=dashed](8,0)(8,1)
\psline[linestyle=dashed](9,0)(9,1)
\end{pspicture}
\end{center}
\caption{$I_\PP$ linearly related}
\label{Polylinres}
\end{figure}
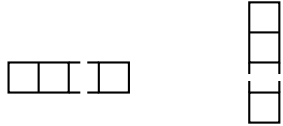

Let $r=\text{rank}~L,$ $q=r-1$ and $p=0,$ then the upper bounded Hibi subring $R[p,q;L]$ has the defining ideal determined by a polyomino of one of the forms given in Figure~\ref{Polylinres}. Without loss of generality, we may assume that $I^{0,r-1}_L=I_{\MP}$ where ${\MP}$ is the polyomino consisting of the cells $[a,a+(1,1)]$, where $a\in\{(0,0),(1,0),\dots,(s-1,0)\}$ for some $s\geq 2$. Since $\text{rank}~L=q+1,$ $L$ contains exactly one of the cells $\alpha=[(s,0),(s,0)+(1,1)]$ and $\beta=[(s-1,1),(s-1,1)+(1,1)]$. If $L$ contains the cell $\alpha$, then $L$ is of the form $(i)$. Otherwise, $L$ contains the cell $\beta$. If $s>2,$ then we may choose the Hibi subring $R[1,\text{rank}~L;L]$ which does not have a linear resolution by \cite[Theorem 4.1]{EHH}. Hence $s=2$ and $L=\MI(P)$ is of one of the forms $(ii)$ and $(iii)$.

The converse is obvious.
\end{proof}

The above theorem shows that if $I_L$ has a linear resolution, then $I^{p,q}_L$ has a linear resolution as well for any
$0\leq p<q\leq\rank L.$ We are now interested to see whether the property of $I_L$ of being linearly related is inherited by all $I^{p,q}_{L}.$ The following example shows that this is not the case. Before discussing it, let us recall some facts. The ideal  $I_L$ is linearly related if its relation module, $\Syz_1(I_L)$ is generated only in degree $3.$ The planar lattices $L$ whose ideal $I_L$ is linearly related are characterized in \cite[Theorem 3.12]{En}:

\begin{Theorem}\cite{En}\label{linrel}
Let $L$ be a planar distributive lattice, $L\subseteq [(0,0),(m,n)]$ with $m,n\geq 2$. The ideal $I_L$ is linearly related if and only if  the following conditions hold:
\begin{itemize}
	\item [(i)] At most one of the vertices $(m,0)$ and $(0,n)$ does not belong to $L.$
	\item [(ii)] The vertices $(1,n-1)$ and $(m-1,1)$ belong to $L.$
\end{itemize}
\end{Theorem}

In \cite[Theorem 3.1]{EHH}, the polyominoes whose associated binomial ideals are linearly related are characterized. We include here the
complete statement for the convenience of the reader. Theorem~\ref{main} refers to Figure~\ref{shape}.  We  assume that $[(0,0),(m,n)]$ is the smallest interval with the property that $V(\MP)\subseteq [(0,0),(m,n)]$. Here $V(\MP)$ denotes the set of all the vertices of $\MP$. The elements $(0,0), (m,0), (0,n)$ and $(m,n)$ are the corners of $\MP$.
The corners $(0,0), (m,n)$ respectively $(m,0),(0,n)$ are called opposite corners.

\begin{figure}[hbt]
\begin{center}
\psset{unit=0.3cm}
\begin{pspicture}(6,-1)(6,12)
{
\pspolygon(0,2)(0,6)(1,6)(1,9)(4,9)(4,10)(10,10)(10,9)(12,9)(12,7)(13,7)(13,3)(12,3)(12,0)(11,0)(11,-1)(4,-1)(4,0)(1,0)(1,2)
}
\rput(-0.5,-0.2){$(1,1)$}
\rput(14.5,-0.2){$(m-1,1)$}
\rput(-1.4,9.2){$(1,n-1)$}
\rput(15.5,9.2){$(m-1,n-1)$}
\rput(1,0){$\bullet$}
\rput(1,9){$\bullet$}
\rput(12,9){$\bullet$}
\rput(12,0){$\bullet$}

\rput(4,-1.6){$i_1$}
\rput(11.2,-1.6){$i_2$}
\rput(4,10.7){$i_3$}
\rput(10.2,10.7){$i_4$}
\rput(-0.7,2.2){$j_1$}
\rput(-0.7,6.1){$j_2$}
\rput(13.7,6.8){$j_4$}
\rput(13.7,3){$j_3$}
\end{pspicture}
\end{center}
\caption{Possible shapes for a linearly related polyomino}\label{shape}
\end{figure}
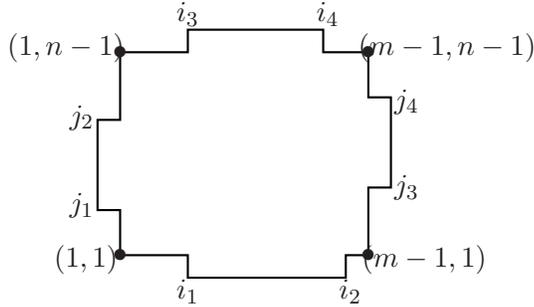

In Figure~\ref{shape}, the number $i_1$ is also allowed to be $0$ in which case also $j_1=0$. In this case, the polyomino contains the corner $(0,0)$. A similar  convention applies to the other corners. In Figure~\ref{shape}, all four corners $(0,0), (0,n), (m,0)$ and $(m,n)$ are missing from the polyomino.

\begin{Theorem}\cite{EHH}
\label{main}
Let $\MP$ be a convex polyomino. The following conditions are equivalent:
\begin{enumerate}
\item[{\em (a)}] $\MP$ is linearly related;
\item[{\em (b)}] $I_\MP$ admits no Koszul relation pairs;
\item[{\em (c)}] Let, as we may assume, $[(0,0),(m,n)]$ be the smallest interval with the property that $V(\MP)\subseteq [(0,0),(m,n)]$. Then $\MP$ has the shape as displayed in Figure~\ref{shape}, and one of the following conditions holds:
    \begin{enumerate}
    \item[{\em (i)}] at most one of the  corners does not belong to $V(\MP)$;
    \item[{\em (ii)}] two of the corners do not belong to $V(\MP)$,  but they are not opposite to each other. In other words, the missing corners are not the corners  $(0,0),(m,n)$ or the corners $(m,0),(0,n)$.
\item[{\em (iii)}] three of the corners do not belong to $V(\MP)$. If the missing corners are $(m,0)$, $(0,n)$ and $(m,n)$ (which one may assume without loss  of generality), then referring to Figure~\ref{shape} the following conditions must be satisfied: either $i_2=m-1$ and $j_4\leq j_2$, or $j_2=n-1$ and $i_4\leq i_2$.
    \end{enumerate}
\end{enumerate}
\end{Theorem}

With all these tools at hand, we can move on to our study.

\begin{Example}\label{notlinearrel}{\em
Let $L$ be the planar distributive lattice whose poset of join-irreducible elements consists of two disjoint chains of lengths $m,$ respectively $n.$
Then, as a planar lattice, $L$ consists of all the lattice points $(i,j)$ with $0\leq i\leq m$ and $0\leq j\leq n. $ In
Figure~\ref{notlinrel}, we displayed such a lattice for $m=5$ and $n=4.$

\begin{figure}[hbt]
\begin{center}
\psset{unit=0.4cm}
\begin{pspicture}(0,0)(5,4)

\pspolygon(0,0)(5,0)(5,4)(0,4)
\psline(0,1)(5,1)
\psline(0,2)(5,2)
\psline(0,3)(5,3)
\psline(1,0)(1,4)
\psline(2,0)(2,4)
\psline(3,0)(3,4)
\psline(4,0)(4,4)
\psline[linestyle=dashed](0,3)(3,0)
\psline[linestyle=dashed](3,4)(5,2)
\end{pspicture}
\end{center}
\caption{ Bounded subrings which are not linearly related}
\label{notlinrel}
\end{figure}
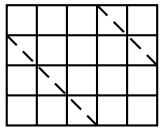

If $m,n\geq 2,$ then
the defining ideal of any rank bounded Hibi subring $R[p,q;L]$ with $0<p<q<\rank L=m+n$ is not linearly related. This is due to the fact that, for $0<p<q<\rank L$, the ideal $I^{p,q}_L$ is actually the ideal of a polyomino $\MP$ whose opposite corners are missing, thus
$\MP$ does not satisfy the conditions of Theorem~\ref{main}.

On the other hand, if one considers the upper-bounded subring $R[0,q;L]$ with $q=m+n-1$, its defining ideal is linearly related since the corresponding polyomino satisfies condition (i) in Theorem~\ref{main}.
 }
\end{Example}

The next theorem refers to Figure~\ref{shape2}. In this figure, the number $i_{2}$ is allowed to be $m$ in which case also $j_{2}=0$. In this case, the polyomino contains the corner $(m,0).$ A similar convention applies to the corner$(0,n).$
\begin{figure}[hbt]
\begin{center}
\psset{unit=0.3cm}
\begin{pspicture}(6,-1)(6,12)
{
\pspolygon(0,-1)(0,6)(1,6)(1,9)(4,9)(4,10)(13,10)(13,3)(12,3)(12,0)(11,0)(11,-1)(0,-1)
}
\rput(-1.5,-1.2){$(0,0)$}
\rput(14.5,-0.2){$(m-1,1)$}
\rput(-1.4,9.2){$(1,n-1)$}
\rput(14.7,10.2){$(m,n)$}
\rput(0,-1){$\bullet$}
\rput(1,9){$\bullet$}
\rput(13,10){$\bullet$}
\rput(12,0){$\bullet$}

\rput(11.2,-1.6){$i_2$}
\rput(4,10.7){$i_1$}
\rput(-0.7,6.1){$j_1$}
\rput(13.7,3){$j_2$}
\psline[linestyle=dashed](0,1)(2,-1)
\psline[linestyle=dashed](11,10)(13,8)
\end{pspicture}
\end{center}
\caption{}\label{shape2}
\end{figure}

\begin{Theorem}\label{linrelth}
Let $m,n\geq 2$ and $L\subseteq [(0,0),(m,n)]$ be a planar distributive lattice with the property that $I_L$ is linearly related. Let $p,q$ be integers such that $0\leq p<q\leq \rank L$. Then $I^{p,q}_L$ is linearly related as well if and only if one of the following conditions is satisfied:
\begin{itemize}
	\item [(a)] Both of the corners $(0,n)$ and $(m,0)$ belong to $L.$ In this case, $(p,q)$ may be any pair in the following set:

$\{(0,m+n-2),(0,m+n-1),(0,m+n),(1,m+n),(2,m+n)\}$
    \item [(b)] Exactly one of the corners $(0,n)$ and $(m,0)$ does not belong to $L.$ If the missing corner is $(0,n)$ (we can state conditions analogy to this when the other corner is missing), then referring to Figure~\ref{shape2}, $(p,q)$ may be any pair in the following sets:
\begin{itemize}
	\item [(1)] $\{(0,m+n)\}$
	\item [(2)] $\{(1,m+n),(2,m+n),(0,m+n-1),(0,m+n-2)\}$
    \item [(3)] $\{(1,m+n-1)~\text{if}~ j_{1}= n-1 ~\text{or} ~i_{1}= 1\}\cup\{(1,m+n-1),(1,m+n-2)~\text{if} ~j_{1}< n-1\}\cup\{(1,m+n-1),(2,m+n-1)~\text{if} ~i_{1}>1\}$
\end{itemize}
\end{itemize}
\end{Theorem}

\begin{proof}
Let $L$ be a planar distributive lattice such that $I_{L}$ is linearly related then $L$ is one of the following forms:

\begin{figure}[hbt]
\begin{center}
\psset{unit=0.4cm}
\begin{pspicture}(0,0)(14,3)
\pspolygon(0,0)(4,0)(4,3)(0,3)
\pspolygon(5,0)(5,1.7)(5.3,1.7)(5.3,2.7)(6.3,2.7)(6.3,3)(9,3)(9,0)(5,0)
\pspolygon(10,0)(10,3)(14,3)(14,1.3)(13.7,1.3)(13.7,0.3)(12.7,0.3)(12.7,0)(10,0)
\rput(5.3,2.7){$\bullet$}
\rput(13.7,0.3){$\bullet$}
\end{pspicture}
\end{center}
\caption{}
\label{Llinrel}
\end{figure}

Let $\MP$ be the convex polyomino such that $I_{\MP}$ is the polyomino ideal which corresponds to the defining ideal $I^{p,q}_{L}$ of the rank bounded Hibi subring $R[p,q;L]$.
As Theorem~\ref{main} states all the possible shapes of linearly related polyominoes, we can derive conditions for $p$ and $q$ by making use of this theorem. We know that $\MP$ should contain all the vertices $(1,1)$, $(m-1,1)$, $(m-1,n-1)$ and $(1,n-1)$ if $I_{\MP}$ is linearly related.

 Let $L$ be a lattice of the form as displayed on the left of Figure~\ref{Llinrel}. Then we do not have the choice to miss both of the corners $(0,0)$ and $(m,n)$ in the same time because they are opposite corners. If we miss $(0,0)$ then $q=m+n$ and we can take $0<p\leq 2$ and if we miss $(m,n)$ then we have $p=0$ and can take $m+n-2\leq q<m+n$.

 Let $L$ be a lattice of the form as displayed in the middle of Figure~\ref{Llinrel}, Then we have a few choices of the corners of $\MP$:

Case (1). We miss none of the corners $(0,0)$ and $(m,n).$ In this case, $p=0$ and $q=m+n$.

Case (2). We miss exactly one of the corners $(0,0)$ and $(m,n).$ If the missing corner is $(0,0)$, then $q=m+n$ and we have two choices for $p$, either $p=1$ or $p=2.$ We can not take $p>2$ because then $(1,1)$ will no longer be a vertex of $\MP$. In a similar way, if we miss only $(m,n)$ then $p=0$ and we have two choices for $q$, either $q=m+n-2$ or $q=m+n-1$. Again, we can not take $q<m+n-2$ because then $\MP$ will miss the vertex $(m-1,n-1)$ as well.

Case (3). We can miss even both of the corners $(0,0)$ and $(m,n).$ But in this case, we must put some extra conditions. The first choice is to fix $p=1$. Then referring to Figure~\ref{shape2}, if $j_{1}=n-1$ then we can take $q=m+n-1$ only. And if $j_{1}<n-1$ then we have again two choices for $q$, either $q=m+n-1$ or $q=m+n-2$ in order to assure that $j_{1}\leq q-m$ and that $(m-1,n-1)$ belongs to $V(\MP)$. The other choice is to fix $q=m+n-1$. Then in the similar way, if $i_{1}=1$ we can take $p=1$ only. And if $i_{1}>1$ then we have two choices for $p$, either $p=1$ or $p=2$ in order to assure that $i_{1}\geq p$ and that $(1,1)$ is a vertex of $\MP$.
\end{proof}

{}

\end{document}